\documentclass[11pt,twoside,reqno,centertags,draft]{amsart}
\usepackage{amsfonts}
\usepackage{color,enumitem,graphicx}
\usepackage[colorlinks=true,urlcolor=blue,
citecolor=red,linkcolor=blue,linktocpage,pdfpagelabels,
bookmarksnumbered,bookmarksopen]{hyperref}

  \usepackage{amsmath,amsthm,amsfonts,amssymb}

\begin{document}

\title{fractional Hardy-Sobolev elliptic problems}
\date{}
\maketitle

\vspace{ -1\baselineskip}

{\small
\begin{center}
  {\sc  Jianfu Yang} \\
Department of Mathematics,
Jiangxi Normal University\\
Nanchang, Jiangxi 330022,
P.~R.~China\\
email: jfyang\_2000@yahoo.com\\[10pt]
{\sc  Xiaohui Yu} \\
The Center for China's Overseas Interests,
Shenzhen University\\
Shenzhen Guangdong, 518060,
P.~R.~China\\
email: yuxiao\_211@163.com\\[10pt]

\end{center}
}

\renewcommand{\thefootnote}{}

\begin{quote}
{\bf Abstract.} In this paper, we study the following singular nonlinear elliptic problem
\begin{equation}\label{eq:1}
 \left\{
  \begin{array}{ll}
  \displaystyle
(-\Delta)^{\frac \alpha 2} u=\lambda
|u|^{r-2}u+\mu\frac{|u|^{q-2}u}{|x|^{s}}\quad &{\rm in }\quad \Omega,    \\
\\ u=0 &{\rm on }\quad
\partial\Omega,
\end{array}
\right.
 \end{equation}
where $\Omega$ is a smooth bounded domain in $\mathbb R^N$ with
$0\in \Omega$, $\lambda,\mu>0,0<s\leq\alpha$, $(-\Delta)^{\frac
\alpha 2}$ is the fractional Laplacian operator with $0<\alpha<2$. We establish existence results of problem \eqref{eq:1} for
subcritical, Sobolev critical and Hardy-Sobolev critical cases.
\end{quote}

\newcommand{\N}{\mathbb{N}}
\newcommand{\R}{\mathbb{R}}
\newcommand{\Z}{\mathbb{Z}}

\newcommand{\cA}{{\mathcal A}}
\newcommand{\cB}{{\mathcal B}}
\newcommand{\cC}{{\mathcal C}}
\newcommand{\cD}{{\mathcal D}}
\newcommand{\cE}{{\mathcal E}}
\newcommand{\cF}{{\mathcal F}}
\newcommand{\cG}{{\mathcal G}}
\newcommand{\cH}{{\mathcal H}}
\newcommand{\cI}{{\mathcal I}}
\newcommand{\cJ}{{\mathcal J}}
\newcommand{\cK}{{\mathcal K}}
\newcommand{\cL}{{\mathcal L}}
\newcommand{\cM}{{\mathcal M}}
\newcommand{\cN}{{\mathcal N}}
\newcommand{\cO}{{\mathcal O}}
\newcommand{\cP}{{\mathcal P}}
\newcommand{\cQ}{{\mathcal Q}}
\newcommand{\cR}{{\mathcal R}}
\newcommand{\cS}{{\mathcal S}}
\newcommand{\cT}{{\mathcal T}}
\newcommand{\cU}{{\mathcal U}}
\newcommand{\cV}{{\mathcal V}}
\newcommand{\cW}{{\mathcal W}}
\newcommand{\cX}{{\mathcal X}}
\newcommand{\cY}{{\mathcal Y}}
\newcommand{\cZ}{{\mathcal Z}}

\newcommand{\abs}[1]{\lvert#1\rvert}
\newcommand{\xabs}[1]{\left\lvert#1\right\rvert}
\newcommand{\norm}[1]{\lVert#1\rVert}

\newcommand{\loc}{\mathrm{loc}}
\newcommand{\p}{\partial}
\newcommand{\h}{\hskip 5mm}
\newcommand{\ti}{\widetilde}
\newcommand{\D}{\Delta}
\newcommand{\e}{\epsilon}
\newcommand{\bs}{\backslash}
\newcommand{\ep}{\emptyset}
\newcommand{\su}{\subset}
\newcommand{\ds}{\displaystyle}
\newcommand{\ld}{\lambda}
\newcommand{\vp}{\varphi}
\newcommand{\wpp}{W_0^{1,\ p}(\Omega)}
\newcommand{\ino}{\int_\Omega}
\newcommand{\bo}{\overline{\Omega}}
\newcommand{\ccc}{\cC_0^1(\bo)}
\newcommand{\iii}{\opint_{D_1}D_i}

\theoremstyle{plain}
\newtheorem{Thm}{Theorem}[section]
\newtheorem{Lem}[Thm]{Lemma}
\newtheorem{Def}[Thm]{Definition}
\newtheorem{Cor}[Thm]{Corollary}
\newtheorem{Prop}[Thm]{Proposition}
\newtheorem{Rem}[Thm]{Remark}
\newtheorem{Ex}[Thm]{Example}

\numberwithin{equation}{section}
\newcommand{\meas}{\rm meas}
\newcommand{\ess}{\rm ess} \newcommand{\esssup}{\rm ess\,sup}
\newcommand{\essinf}{\rm ess\,inf} \newcommand{\spann}{\rm span}
\newcommand{\clos}{\rm clos} \newcommand{\opint}{\rm int}
\newcommand{\conv}{\rm conv} \newcommand{\dist}{\rm dist}
\newcommand{\id}{\rm id} \newcommand{\gen}{\rm gen}
\newcommand{\opdiv}{\rm div}

\vskip 0.2cm \arraycolsep1.5pt
\newtheorem{Lemma}{Lemma}[section]
\newtheorem{Theorem}{Theorem}[section]
\newtheorem{Definition}{Definition}[section]
\newtheorem{Proposition}{Proposition}[section]
\newtheorem{Remark}{Remark}[section]
\newtheorem{Corollary}{Corollary}[section]

\section {Introduction}

The main objective of this paper is to consider the following fractional Hardy-Sobolev elliptic problem
\begin{equation}\label{eq:1.1}
\left\{
  \begin{array}{ll}
  \displaystyle
(-\Delta)^{\frac \alpha 2} u=\lambda
|u|^{r-2}u+\mu\frac{|u|^{q-2}u}{|x|^{s}}\quad &{\rm in }\quad \Omega,    \\
\\ u=0 &{\rm on }\quad
\partial\Omega,
\end{array}
\right.
 \end{equation}
where $\Omega$ is a smooth bounded domain in $\mathbb R^N$ with
$0\in \Omega$, $\lambda,\mu>0,\,0<s\leq\alpha$. We define fractional
Laplacian operator $(-\Delta)^{\frac \alpha 2}$ with $0<\alpha<2$ as
follows.

Let $\{(\lambda_k,\varphi_k)\}^\infty_{k=1}$ be the eigenvalues and corresponding eigenfunctions of the
Laplacian operator $-\Delta$ in $\Omega$ with zero Dirichlet boundary values on $\partial\Omega$ normalized by $\|\varphi_k\|_{L^2(\Omega)} = 1$, i.e.
\[
-\Delta \varphi_k = \lambda_k \varphi_k\quad{\rm in}\ \Omega;\quad \varphi_k = 0\quad{\rm on}\ \partial\Omega.
\]
We define the space $H^{\frac\alpha2}_0(\Omega)$ by
\[
H^{\frac\alpha2}_0(\Omega)=\Big\{u=\sum_{k=1}^\infty u_k\varphi_k\in L^2(\Omega): \sum_{k=1}^\infty \lambda_k^{\frac \alpha2}u_k^2<\infty\Big\},
\]
which is equipped with the norm
\[
\|u\|_{H^{\frac\alpha2}_0(\Omega)} =\bigg(\sum_{k=1}^\infty \lambda_k^{\frac \alpha2}u_k^2\bigg)^{\frac 12}.
\]
For any $u\in H^{\frac\alpha2}_0(\Omega)$, the fractional Laplacian $(-\Delta)^{\frac \alpha 2}$ is defined by
\[
(-\Delta)^{\frac \alpha 2}u = \sum_{k=1}^\infty \lambda_k^{\frac \alpha 2}u_k\varphi_k,
\]
which is a nonlocal operator. In \cite{CS}, by using Dirichlet-Neumann mapping, elliptic problems with fractional Laplacian in $\mathbb{R}^N$ can be converted into local elliptic problems. This argument was applied to bounded domains in \cite{BCP,BCPa,CaT,CDDS,T}. Precisely, let $\mathcal{C}_{\Omega}=\Omega\times(0,\infty)$. We define
\[
H^1_{0,L}(\mathcal{C}_\Omega)
= \{w\in L^2(\mathcal{C}_\Omega): w = 0 \,\,{\rm on}\,\, \partial_L\mathcal{C}_\Omega,\  \kappa_\alpha\int_{\mathcal{C}_\Omega}y^{1-\alpha}|\nabla w|^2\,dxdy<\infty\},
\]
which is a Hilbert space with the norm
\[
\|w\|_{H^1_{0,L}(\mathcal{C}_\Omega)}^2=\kappa_\alpha\int_{\mathcal{C}_\Omega}y^{1-\alpha}|\nabla w|^2\,dxdy.
\]
For any $u\in H^{\frac\alpha2}_0(\Omega)$, the unique solution $w\in H^1_{0,L}(\mathcal{C}_\Omega)$ of the problem
\begin{equation}\label{eq:1.2}
\left\{ \arraycolsep=1.5pt
\begin{array}{ll}
-{\rm div}(y^{1-\alpha}\nabla w) = 0, & \ \ \mathcal{C}_{\Omega},\\[1mm]
w=0,& \ \
\partial_L\mathcal{C}_{\Omega}=\partial\Omega\times(0,\infty),\\[1mm]
w = u , &\ \ \Omega\times\{0\},
 \end{array}
 \right.
 \end{equation}
is referred to be the extension $w = E_\alpha(u)$ of $u$. We know from \cite{BCP,BCPa,CDDS} that the mapping $E_\alpha: H^{\frac\alpha 2}_0(\Omega) \to H^1_{0,L}(\mathcal{C}_\Omega)$ is an isometric isomorphism, and we have
\begin{equation}\label{eq:1.2a}
\|u\|_{H^{\frac\alpha 2}_0(\Omega)} = \|E_\alpha(u)\|_{H^1_{0,L}(\mathcal{C}_\Omega)}.
 \end{equation}
It was shown in \cite{CS}, see also \cite{BCPa} etc, that $$-\kappa_\alpha\lim_{y\to 0^+}y^{1-\alpha}\frac {\partial w}{\partial y} = (-\Delta)^{\frac \alpha 2}u.$$
Hence, the restriction of solutions of problem \eqref{eq:1.2} in $\Omega$ are solutions of problem \eqref{eq:1.1}.
Using this sort of extension, one can study the existence of  elliptic problems with the fractional Laplacian by the variational method, see \cite{BCP}, \cite{BCPa}, \cite{CaT} and \cite{T} etc for related results.
Now, we may reformulate the nonlocal problem (\ref{eq:1.1}) in a local way, that is,
\begin{equation}\label{eq:1.3}
  \left\{
  \begin{array}{ll}
  \displaystyle
div(y^{1-\alpha}\nabla w)=0 \  \  &{\rm in}\quad C_\Omega,    \\[1mm]
w=0,\ \
&{\rm on}\quad\partial_L\mathcal{C}_{\Omega},\\[1mm]
\lim_{y\to 0}y^{1-\alpha}\frac{\partial w}{\partial \nu}=\lambda |w|^{r-2}w+\mu\frac{|w|^{q-2}w}{|x|^s} \ \ &{\rm in }\quad
\Omega,
\end{array}
\right.
\end{equation}
where $\frac {\partial }{\partial \nu}$ is the outward normal derivative. We then turn to study the equivalent problem (\ref{eq:1.3}). Define on $H^1_{0,L}(\mathcal{C}_\Omega)$, the functional
\begin{equation}\label{eq:1.3b}
I(w)=\frac 12\int_{\mathcal{C}_\Omega} y^{1-\alpha}|\nabla w(x,y)|^2\,dxdy-\frac
{\lambda}r\int_{\Omega} |w|^{r}\,dx-\frac \mu q\int_{\Omega}
\frac{|w|^q}{|x|^s}\,dx.
\end{equation}
Critical points of $I(w)$ are weak solutions of \eqref{eq:1.3}.

In the case $\alpha = 2$, problem \eqref{eq:1.1} with Hardy-Sobolev
term has been extensively studied, see for instance \cite{CHP, GY, KP, KP1, LP} and related
references. If $0<\alpha<2,\,s = 0$ and $q =
2_\alpha^*=\frac{2N}{N-\alpha}$, problem \eqref{eq:1.1} is
Brezis-Nirenberg problem with fractional Laplacian. In \cite{T},
existence results were found for the case $\alpha =1$, while the
general case was considered in \cite{BCP}. The main difficulty of
such a problem is the lack of compactness. As in \cite{BN}, the
compactness can be retained below levels related to the best
constant $S_{\alpha}^{E}$ of the trace inequality, that is the
Palais-Smale condition holds for values below these levels, where
\[
S_{\alpha}^E = \inf_{w\in
H^1_{0,L}(\mathcal{C}_\Omega)}\frac{\int_{\mathcal{C}_\Omega}y^{1-\alpha}|\nabla
w(x,y)|^2\,dxdy}{\Big(\int_{\Omega}|w(x,0)|^{2^*_\alpha}\,dx\Big)^{\frac
2{2^*_\alpha}}}.
\]
The constant $S_{\alpha}^{E}$ is related to the best constant
$S_\alpha$ of fractional Sobolev inequality:
\[
S_{\alpha} = \inf_{u\in H_0^{\frac \alpha 2}(\Omega)}\frac{\int_{\Omega}|(-\Delta)^{\frac \alpha 4} u(x)|^2\,dx}{\Big(\int_{\Omega}|u(x)|^{2^*_\alpha}\,dx\Big)^{\frac 2{2^*_\alpha}}}.
\]
In fact, by \cite{BCPa}, $S_{\alpha}^{E} = k_\alpha S_{\alpha}$ with
the constant $k_\alpha$ given there. In the case $\Omega =
\mathbb{R}^N$, it is known from \cite{CLO} that $S_\alpha$ is
achieved by functions with the form
\begin{equation}\label{eq:1.3a}
U(x) = U_\varepsilon(x) = \frac{\varepsilon^{\frac{N-\alpha}2}}{(\varepsilon^2 + |x|^2)^{\frac{N-\alpha}2}}
\end{equation}
for $\varepsilon>0$. Furthermore, in this case, $S_{\alpha}^{ E}$ is
also achieved, the minimizer actually is $w = E_\alpha(U)$, see \cite{BCP} for
details. However, the explicit form of $E_\alpha(U)$ has not been
worked out yet. To verify $(PS)_c$ condition, it is necessary to
investigate further properties of $E_\alpha(U)$. By $(PS)_c$ condition for $I$ we mean that any sequence $\{w_n\}\subset
H^1_{0,L}(\mathcal{C}_\Omega)$ such that $I(w_n)\to c$ and $I'(w_n)\to 0$ as $n\to\infty$ contains a convergent
subsequence.

In this paper, we consider the existence of solutions for problem \eqref{eq:1.1}, or equivalently, problem \eqref{eq:1.3} in both subcritical and critical cases.

Denote by $ L^q(\Omega, \frac{1}{|x|^s})$ the weighted $L^p$ space.
Invoking \eqref{eq:1.2a} and results in \cite{Yang}, we see that the inclusion
$$i:H^1_{0,L}(\mathcal{C}_\Omega)\to L^q(\Omega, \frac{1}{|x|^s})$$
is continuous if $2\leq q\leq 2_\alpha^*(s)=\frac{2(N-s)}{N-\alpha}$, and is compact if $2\leq q < 2_\alpha^*(s)$. The exponent $2_\alpha^*(s)$ is called the critical exponent for fractional Hardy-Sobolev inequality.

In subcritical case, that is, $2\leq q< 2_\alpha^*(s),\,\,2<r<2_\alpha^*$, we show that problem \eqref{eq:1.3} possesses infinitely many solutions by critical point theory. Let
\begin{equation}\label{eq:1.3c}
\mu_s=\inf_{w\in H^1_{0,L}(\mathcal{C}_\Omega)}
\frac{\int_{C_\Omega} y^{1-\alpha}|\nabla w|^2\,dxdy}{\int_{\Omega},
\frac{|w(x,0)|^2}{|x|^s}\,dx}.
\end{equation}
Indeed, we have the following results.
\begin{Theorem}\label{thm:1.1}
Suppose that $0\leq s<\alpha$, $2\leq q< 2_\alpha^*(s)$,
$2<r<2_\alpha^*$. If either (i) $q>2$ or (ii) $q=2,\,0<\mu<\mu_s$,
problem \eqref{eq:1.1} has infinitely many solutions.
\end{Theorem}

\bigskip

If $q=2, s=\alpha$, problem \eqref{eq:1.1} is related to the fractional Hardy inequality
\[
\int_{\Omega}\frac{|u(x)|^2}{|x|^{\alpha}}\,dx\leq
C_{\alpha,N}\int_{\Omega}|(-\Delta)^{\frac \alpha 4} u(x)|^2\,dx,
\]
which was investigated in \cite{CT}, \cite{FS} and \cite{Y}, and the best constant $C_{\alpha,N}$ was
 computed. By \eqref{eq:1.2a}, $\mu_\alpha$ in \eqref{eq:1.3c} can also be
worked out. We obtain existence results for subcritical case
$2<r<2^*_\alpha$, and nonexistence result for critical case
$r=2^*_\alpha$. Precisely, we have
\begin{Theorem}\label{thm:1.2}
Suppose that $q=2, s=\alpha$.

(i) If $2<r<2_\alpha^*$ and $0<\mu<\mu_\alpha$, then problem
\eqref{eq:1.1} has infinitely many solutions.

(ii) If $r=2_\alpha^*$ and $\Omega$ is a star-shaped domain, then \eqref{eq:1.1} does not possess nontrivial
solution.
\end{Theorem}

If $r=2_\alpha^*$, where $2_\alpha^*$ is the critical fractional Sobolev exponent, this is critical fractional Sobolev problem with singular perturbation term. We may deduce in spirit of \cite{BCP,BN,T} the following results.

\begin{Theorem}\label{thm:1.4}
Suppose $0<s<\alpha,r=2_\alpha^*$ and $2\leq q<2_\alpha^*(s)$.

(i) If $q>2$, then problem \eqref{eq:1.1} possesses at least a
positive solution provided that
$N>\frac{2(\alpha-s)}{q}+\alpha$.

(ii) If $q=2$ and $0<\mu<\mu_s$, then
problem \eqref{eq:1.1} possesses at least a positive solution provided that $N> 2\alpha-s$.
\end{Theorem}

Finally, if $q= 2^*_\alpha(s) = \frac{2(N-s)}{N-\alpha}$, it is a critical Hardy-Sobolev problem. If $\alpha = 2$, such a problem has been studied by \cite{CHP, GY, KP, KP1, LP} etc. The exponent $2^*_\alpha(s)$ is the critical exponent for the fractional Hardy-Sobolev inequality
\begin{equation}\label{eq:1.4}
\big(\int_{\mathbb{R}^N}\frac{|u(x)|^{2^*_{\alpha}(s)}}{|x|^{s}}\,dx\big)^{\frac
2{2^*_{\alpha}(s)}}\leq C\int_{\mathbb{R}^N}|(-\Delta)^{\frac \alpha
4} u(x)|^2\,dx.
\end{equation}
The description of the compactness for critical Hardy-Sobolev problems is closely related to the best constant $S_{s,\alpha}$ for  the fractional Hardy-Sobolev inequality:
\begin{equation}\label{eq:1.5}
S_{s,\alpha} = \inf_{u\in \dot H^{\frac \alpha 2}(\mathbb{R}^{N}),
u\not\equiv 0}\frac{\int_{\mathbb{R}^N}|(-\Delta)^{\frac \alpha 4}
u(x)|^2\,dx}{\big(\int_{\mathbb{R}^N}\frac{|u(x)|^{2^*_{\alpha}(s)}}{|x|^{s}}\,dx\big)^{\frac
2{2^*_{\alpha}(s)}}},
\end{equation}
where the space $\dot H^{\frac \alpha 2}(\mathbb{R}^{N})$ is defined
as the completion of $C^\infty_0(\mathbb{R}^{N})$ under the norm
\[
\|u\|_{\dot H^s(\mathbb{R}^{N})}^2 =
\int_{\mathbb{R}^N}|(-\Delta)^{\frac \alpha 4} u(x)|^2\,dx.
\]
A minimizer of the minimization problem $S_{s,\alpha}$ will be used as usually, to verify $(PS)_c$ condition for critical Hardy-Sobolev problems. It was proved by \cite{Y} that $S_{s,\alpha}$ is achieved by a positive, radially symmetric and non-increasing function.
But, an explicit formula as \eqref{eq:1.3a} for the minimizer has not been found yet. Hence, we need to find
further properties of the minimizer $u$, such as decaying laws, to verify $(PS)_c$ condition. This will be done in section 3 by the Kelvin transform and elliptic regularity theory. Furthermore, estimates for $|\nabla u|$ are also needed. Since $u$ can be expressed be the Riesz potential, the decaying law of $|\nabla u|$ at infinity can be established by the decay law of $u$. However, $|\nabla u|$ is possibly singular at the origin, it is required an estimate of singular order of $|\nabla u|$ at the origin. These are shown in section 3.  Eventually, we obtain the following results for critical Hardy-Sobolev problem.

\begin{Theorem}\label{thm:1.3}

Suppose that $0<s<\alpha, \, q=2^*_{\alpha}(s)$.

(i) If $2<r<2_\alpha^*$, then problem \eqref{eq:1.1} possesses at least a positive solution provided that $N>\alpha+\frac{2\alpha}r$.

(ii) If $r=2$ and $0<\lambda<\lambda_1$, then problem \eqref{eq:1.1}
possesses at least a positive solution provided that $N\geq
2\alpha$, where $\lambda_1$ is the first eigenvalue of
$(-\Delta)^{\frac \alpha 2}$ in $\Omega$.

(iii) If $r=2_\alpha^*$ and
$\Omega$ is a star-shaped domain, then problem \eqref{eq:1.1} does not possess
nontrivial solution for any $\mu,\lambda>0$.

\end{Theorem}

This paper is organized as follows. In section 2, we study the
multiple solution results for subcritical problems and nonexistence
results for critical problems. In section 3, we study the
existence result for problem \eqref{eq:1} with $r=2_\alpha^*$ and
$q=2_\alpha^*(s)$ respectively.

\bigskip

\section {Subcritical problem and nonexistence}

\bigskip

In this section, we first show that in the subcritical case, problem
\eqref{eq:1.3} has infinitely many solutions, then we prove some
nonexistence results. For this purpose, we will find critical points
of the functional $I$ defined by \eqref{eq:1.3b} by the following
minimax theorem, which is Theorem 9.12 in \cite{Ra}.

\begin{Lemma}\label{lem:2.1}
Let $E$ be an infinite dimensional Banach space and let $I\in
C^2(E,\mathbb R)$ be even, satisfying $(PS)$ condition, and
$I(0)=0$. If $E=V\oplus
X$, where $V$ is finite dimensional, and $I$ satisfies\\
(i) there exist constants $\rho,\alpha>0$, such that $I_{\partial
B_\rho\cap X}\geq \alpha$, and\\
(ii) for each finite dimensional subspace $\tilde E\subset E$, there
is an $R=R(\tilde E)$ such that $I\leq 0$ on $\tilde E\setminus
B_{R(\tilde E)}$,\\
then $I$ possesses an unbounded sequence of critical values.
\end{Lemma}

Now we prove Theorem \ref{thm:1.1} and $(i)$ of  Theorem
\ref{thm:1.2}. In the case $q=2$ in Theorem \ref{thm:1.1} and
Theorem \ref{thm:1.2}, since $0<\mu<\mu_s$, we remark that the norm
$$
\|w\|= \Big(\int_{\mathcal{C}_\Omega}y^{1-\alpha}|\nabla w|^2\,dxdy-\mu\int_{\Omega}\frac{w(x,0)^2}{|x|^s}\,dx\Big)^{\frac
12}
$$
is equivalent to the norm $\|w\|_{H^1_{0,L}(\mathcal{C}_\Omega)}$ on $H^1_{0,L}(\mathcal{C}_\Omega)$.
\bigskip

\begin{proof}[{\bf Proof of Theorem \ref{thm:1.1} and $(i)$ of  Theorem \ref{thm:1.2}:}] \ \ We need to verify that
the functional $I$ satisfies the conditions in Lemma \ref{lem:2.1},
the conclusions then follow. We prove Theorem \ref{thm:1.1} first,
then we sketch the proof of (i) in Theorem \ref{thm:1.2}.

First, we show that $(PS)$ condition holds for $I$. Let $\{w_n\}\subset H^1_{0,L}(\mathcal{C}_\Omega)$ be a $(PS)$ sequence, that is,
\[
|I(w_n)|\leq C\quad{\rm and}\quad I'(w_n)\to 0
\]
as $n\to\infty$. For the case $q>2$, this implies that if $r\geq q$,
we have
$$
(\frac 12-\frac 1q)\|w_n\|_{H^1_{0,L}(\mathcal{C}_\Omega)}^2+(\frac
1q-\frac 1r) \lambda\int_{\Omega} |w_n(x,0)|^r\,dx\leq
C+o(1)\|w_n\|_{H^1_{0,L}(\mathcal{C}_\Omega)}.
$$
While if $r<q$, one has
$$
(\frac 12-\frac 1r)\|w_n\|_{H^1_{0,L}(\mathcal{C}_\Omega)}^2+(\frac 1r-\frac 1q) \mu\int_{\Omega}
\frac{|w_n(x,0)|^q}{|x|^s}\,dx\leq C+o(1)\|w_n\|_{H^1_{0,L}(\mathcal{C}_\Omega)}.
$$
These inequalities imply that
$\{\|w_n\|_{H^1_{0,L}(\mathcal{C}_\Omega)}\}$ is bounded. Similarly,
if $q=2$, we have
$$
(\frac 12-\frac 1r)\|w_n\|_{H^1_{0,L}(\mathcal{C}_\Omega)}\leq C+o(1)\|w_n\|_{H^1_{0,L}(\mathcal{C}_\Omega)},
$$
which implies $\|w_n\|_{H^1_{0,L}(\mathcal{C}_\Omega)}$ is bounded.

Now, we show that $\{w_n\}$ has a convergent subsequence. Since
$\{w_n\}$ is bounded, we can suppose that, up to a subsequence,
$w_n\rightharpoonup w$ in $H^1_{0,L}(\mathcal{C}_\Omega)$ as $n\to\infty$. By the Sobolev compact imbedding theorem, we have
$$
\int_{\Omega} |w_n(x,0)|^r\,dx\to \int_{\Omega} |w(x,0)|^r\,dx
$$
and
$$
\int_{\Omega} \frac{|w_n(x,0)|^q}{|x|^s}\,dx\to \int_{\Omega}
\frac{|w(x, 0)|^q}{|x|^s}\,dx
$$
as $n\to\infty$. Therefore,
$$
\|w_n\|_{H^1_{0,L}(\mathcal{C}_\Omega)}^2\to \lambda\int_{\Omega} |w(x,0)|^r\,dx+\mu\int_{\Omega}
\frac{|w(x,0)|^q}{|x|^s}\,dx=\|w\|_{H^1_{0,L}(\mathcal{C}_\Omega)}^2
$$
as $n\to\infty$. The convergence $w_n\to w$ strongly in $H^1_{0,L}(\mathcal{C}_\Omega)$ follows from the fact that $w_n\rightharpoonup w$ and $\|w_n\|_{H^1_{0,L}(\mathcal{C}_\Omega)}\to \|w\|_{H^1_{0,L}(\mathcal{C}_\Omega)}$.

Next, let $V=\emptyset$ and $E=X=H^1_{0,L}(\mathcal{C}_\Omega)$ in Lemma \ref{lem:2.1}. We claim
that there exist constants $\rho,\alpha>0$, such that $I_{\partial
B_\rho\cap X}\geq \alpha$. Indeed, if $q>2$, by the Sobolev embedding theorem,
$$
I(w)\geq \frac 12\|w\|_{H^1_{0,L}(\mathcal{C}_\Omega)}^2-C\|w\|_{H^1_{0,L}(\mathcal{C}_\Omega)}^r
-C\|w\|_{H^1_{0,L}(\mathcal{C}_\Omega)}^q,
$$
and if $q=2$, we have
$$
I(w)\geq C_1\|w\|_{H^1_{0,L}(\mathcal{C}_\Omega)}^2-C_2\|w\|_{H^1_{0,L}(\mathcal{C}_\Omega)}^r.
$$
Therefore, we can choose $\alpha,\rho>0$ small enough, such that
$$
I_{\partial B_\rho\cap X}\geq \alpha>0.
$$

Finally, let $\tilde E\subset H^1_{0,L}(\mathcal{C}_\Omega)$ be a finite dimensional subspace. We claim that there
exists $R=R(\tilde E)$ such that $I\leq 0$ on $\tilde E\setminus
B_{R(\tilde E)}$.

Since $\tilde E$ is finite dimensional, any norms on $\tilde E$
are equivalent. So we have
$$
I(u)\leq C_1\|u\|_{H^1_{0,L}(\mathcal{C}_\Omega)}^2-C_2\|u\|_{H^1_{0,L}(\mathcal{C}_\Omega)}^r-C_3\|u\|_{H^1_{0,L}(\mathcal{C}_\Omega)}^q
$$
for any $u\in \tilde E$. Since $2<r<\frac{2N}{N-\alpha}$
and $2\leq q<\frac{2(N-s)}{N-\alpha}$, the result follows easily.
Hence, the proof of Theorem \ref{thm:1.1} is completed by Lemma
\ref{lem:2.1}.

The proof of (i) in Theorem \ref{thm:1.2} is similar to the above,
we sketch it. In this case, we note that
$$
\|w\|= \Big(\int_{\mathcal{C}_\Omega}y^{1-\alpha}|\nabla
w|^2\,dxdy-\mu\int_{\Omega}\frac{w(x,0)^2}{|x|^\alpha}\,dx\Big)^{\frac
12}
$$
defines an equivalent norm as $\|w\|_{H_{0,L}^1(C_\Omega)}$. Let
$\{w_n\}\subset H^1_{0,L}(\mathcal{C}_\Omega)$ be a $(PS)$ sequence.
We can deduce as above that $\{w_n\}$ is bounded. So we have
$w_n\rightharpoonup w$ and
$$
\int_{\Omega} |w_n(x,0)|^r\,dx\to \int_{\Omega} |w(x,0)|^r\,dx,
$$
and conclude that $ \|w_n\|\to \|w\|$ as $n\to \infty$.Thus, the $(PS)$
condition holds. The proof that $I$ satisfying (i) and (ii) in Lemma
\ref{lem:2.1} is the same as above, we omit it. So we have proved (i) of
Theorem \ref{thm:1.2}.

\end{proof}

Now, we prove nonexistence of solutions for critical case. It is deduced by Pohozaev identity for problem \eqref{eq:1.3}.

\begin{Lemma}\label{lem:2.2}
Let $w\in H^1_{0,L}(\mathcal{C}_\Omega)$ be a solution of problem \eqref{eq:1.3}. There holds
\begin{eqnarray*}
&&\frac{N-\alpha}2\int_{\mathcal{C}_\Omega} y^{1-\alpha}|\nabla w|^2\,dxdy-\frac {N\lambda}r\int_{\Omega} |w|^r\,dx\\
&-& \frac{\mu(N-s)}{q}\int_{\Omega} \frac{|w|^q}{|x|^{s}}\,dx +
\frac 12\int_{\partial \Omega\times {\mathbb R^+}}\langle
y^{1-\alpha}(x,y),\nu\rangle|\nabla w|^2\,dS=0,
\end{eqnarray*}
where $\nu$ is the unit outward normal.
\end{Lemma}
\begin{proof} The proof is standard, we sketch it. Multiplying equation \eqref{eq:1.3} by $\langle \nabla
w(x,y),(x,y)\rangle$ and integrating by part over $C_R=\Omega\times [0,R]$,
we get
\begin{eqnarray}\label{eq:3.1}
&-&\frac{N-\alpha}2\int_{C_R}y^{1-\alpha}|\nabla w|^2\,dxdy\\
\nonumber &+&\frac 12\int_{\partial \Omega\times [0,R]}\langle
y^{1-\alpha}(x,y),\nu\rangle |\nabla w|^2\,dS  \\ \nonumber &+&\frac
12 \int_{\Omega\times \{R\}}\langle y^{1-\alpha}(x,y),\nu\rangle
|\nabla w|^2\,dS,\\ \nonumber &=&-\frac{N\lambda}{r}\int_{\Omega}
|w|^r\,dx+\frac \lambda r\int_{\partial \Omega}\langle x,\nu\rangle
|w|^r\,dS -\frac{\mu(N-s)}{q}\int_{\Omega}\frac {|w|^q}{|x|^s}\,dx\\
\nonumber &+&\frac \mu q\int_{\partial \Omega}\langle x,\nu\rangle
\frac{|w|^q}{|x|^s}\,dS +\int_{\partial \Omega\times
[0,R]}y^{1-\alpha}\langle \nabla w,\nu\rangle\langle \nabla
w,(x,y)\rangle\,dS\\ \nonumber &+&\int_{\Omega\times
\{R\}}y^{1-\alpha}\langle \nabla w,\nu\rangle\langle \nabla
w,(x,y)\rangle\,dS.
\end{eqnarray}
Since $w\in H^1_{0,L}(\mathcal{C}_\Omega)$, we may find a sequence
$R_n\to \infty$ such that
$$
\int_{\Omega\times \{R_n\}}\langle y^{1-\alpha}(x,y),\nu\rangle
|\nabla w|^2\,dS\to 0
$$
and
$$
\int_{\Omega\times \{R_n\}}y^{1-\alpha}\langle \nabla
w,\nu\rangle\langle \nabla w,(x,y)\rangle\,dS\to 0
$$
as $n\to \infty$. Moreover, the fact $w\equiv 0$ on $\partial
\Omega\times [0,R]$ enable us to deduce
$$
\langle \nabla w(x,y),\nu\rangle\langle \nabla
w,(x,y)\rangle=\langle (x,y),\nu\rangle|\nabla w|^2.
$$
The conclusion follows by setting $R=R_n$ in \eqref{eq:3.1} and let
$n\to \infty$.

\end{proof}

\begin{proof}[{\bf Proof of (ii) of Theorem \ref{thm:1.2} and $(iii)$ of  Theorem \ref{thm:1.3}:}] The proof is a direct
consequence of Lemma \ref{lem:2.2}. We show (ii) of Theorem \ref{thm:1.2} as follows, the other case can be done similarly. In fact, a solution $w$ of \eqref{eq:1.3} satisfies
\begin{equation}\label{eq:3.2}
\int_{\mathcal{C}_\Omega} y^{1-\alpha}|\nabla w|^2\,dxdy=\lambda
\int_{\Omega} |w|^r\,dx+\mu\int_{\Omega} \frac{|w|^q}{|x|^s}\,dx.
\end{equation}
Under the assumptions (ii) of Theorem \ref{thm:1.2}, i.e.,
$r=\frac{2N}{N-\alpha},q=2$ and $s=\alpha$, Lemma \ref{lem:2.2} and
\eqref{eq:3.2} yield
\[
\begin{split}
\frac 12\int_{\partial \Omega\times {\mathbb R^+}}\langle
y^{1-\alpha}(x,y),\nu\rangle|\nabla w|^2\,dS=0
\end{split}
\]
Since $\Omega$ is star-shaped, we conclude that $\nabla w=0$ on
$\partial \Omega \times [0,\infty)$. The unique
continuum theorem implies $w\equiv 0$.
\end{proof}

\bigskip

\bigskip

\section {Critical case}

\bigskip

In this section, we deal with the existence results for critical
problems. For problem \eqref{eq:1.1}, or equivalently, problem
\eqref{eq:1.3}, we consider both the case that nonsingular
nonlinear term $|u|^{r-2}u$ with critical Sobolev exponent $r =
2^*_\alpha$ and the case that singular nonlinear term $\frac{|u|^{q-2}u}{|x|^{s}}$
with Hardy-Sobolev exponent $q = 2^*_\alpha(s)$.

\subsection {Nonsingular term with Sobolev critical exponent}

\ \ We suppose in this subsection that $0< s <\alpha, \, 2\leq q < 2^*_\alpha(s)$. In this case, we may proceed as \cite{BCP}. We will prove
Theorem \ref{thm:1.4} by the mountain pass theorem \cite{Ra}. Let
$C^* = \frac \alpha{2N}( S_\alpha^E)^{\frac N\alpha}$. It can be
shown as Lemma 5 of \cite{BCP} that that the functional $I$
satisfies $(PS)_c$ condition for $c\in(0, C^*)$. We need to verify
that the mountain pass level is below $C^*$.

Denote by $w_\varepsilon = E_\alpha(U_\varepsilon)$ the extension of
$U_\varepsilon$ given in \eqref{eq:1.3a}. We recall that for any
$u\in \dot H^{\frac \alpha 2}(\mathbb{R}^N)$, the extension $w$ of
$u$ has an explicit expression by the Poisson kernel:
\begin{equation}\label{eq:3.1.1}
w(x,y) = P^\alpha_y*u(x) = Cy^\alpha\int_{\mathbb R_+^N} \frac { u(z)}{(|x-z|^2+y^2)^{\frac{N+\alpha}2}}\,dz.
\end{equation}
Let $\varphi_0(t)\in C^\infty(\mathbb{R}_+)$ be a non-increasing
function satisfying that $\varphi_0(t) = 1$ if $0\leq t\leq \frac
12$, and $\varphi_0(t) = 0$ if $t\geq 1$. Choose $\rho>0$ such that
$\overline{ B_\rho^+(0)}\subset\overline{\mathcal{C}_\Omega}$. Denote
\begin{equation}\label{eq:3.1.1a}
\varphi(x,y) = \varphi_\rho(x,y) = \varphi_0(\frac{\rho_{xy}}\rho),
\end{equation}
where $\rho_{xy} = |(x,y)|$. Then, $\varphi w_\varepsilon\in
H^1_{0,L}(\mathcal{C}_\Omega)$.

\begin{Lemma}\label{lem:3.1.1} For $\varepsilon$ small and $N>2\alpha -s$, there hold

(i) $\int_{\mathcal{C}_\Omega} y^{1-\alpha}|\nabla (\varphi
w_\varepsilon)|^2\,dxdy=\|w_\varepsilon\|_{H_{0,L}^1(C_\Omega)}^2+O(\varepsilon
^{N-\alpha})$.

(ii) $\int_{\Omega} |\varphi
u_\varepsilon|^{2_\alpha^*}\,dx=\|u_\varepsilon\|_{L^{2_\alpha^*}}^{2_\alpha^*}+O(\varepsilon^N)$.

(iii) $\int_{\Omega} \frac{|\varphi u_\varepsilon|^{
q}}{|x|^s}\,dx\geq C\varepsilon^{\frac{\alpha-N}{2}q+N-s}$.
\begin{proof}
(i) was proved in \cite{BCP}. We only prove (ii) and (iii). As for (ii), we have
\[
\begin{split}
&\int_{\mathbb R^N}|u_\varepsilon^{2_\alpha^*}-(\varphi u_\varepsilon
)^{2_\alpha^*}|\,dx\\
&\leq C\int_{\mathbb R^N\setminus B_{\frac
\rho2}(0)}\frac{\varepsilon^N}{[\varepsilon^2+|x|^2]^N}\,dx\\&=\int_{\mathbb
R^N\setminus B_{\frac \rho{2\varepsilon}}(0)} \frac
1{[1+|x|^2]^N}\,dx\\&=C\int_{\frac
\rho{2\varepsilon}}^{\infty}s^{-N-1}\,ds\\&=O(\varepsilon^N),
\end{split}
\]
this proves (ii).

Similarly, for (iii), since $N>2\alpha -s$, we have
\[
\begin{split}
&\int_{\mathbb R^N}\frac{(\varphi u_\varepsilon)^{q}}{|x|^s}\,dx\\
&\geq
\int_{B_{\frac \rho2}(0)}\frac{(\varphi
u_\varepsilon)^{q}}{|x|^s}\,dx\\&=C\int_{B_{\frac
\rho2}(0)}\frac{\varepsilon^{\frac{(N-\alpha)q}{2}}}{|x|^s[\varepsilon^2+|x|^2]^{\frac{(N-\alpha)q}{2}}}\,dx\\&
=\varepsilon^{N-s-\frac{(N-\alpha)q}2}\int_{B_{\frac
\rho{2\varepsilon}}(0)}\frac
1{|x|^s[1+|x|^2]^{\frac{(N-\alpha)q}2}}\,dx\\&\geq
C\varepsilon^{N-s-\frac{(N-\alpha)q}2}.\\
\end{split}
\]
\end{proof}

\end{Lemma}

\begin{proof}[{\bf Proof of Theorem \ref{thm:1.4}:}]
\ \ Without loss of generality, we assume that $\lambda=1$ . First, we treat the case (i).
 We may verify that the function
$$
\bar f(t)= \frac {t^2}2\int_{\mathbb R_+^{N+1}} |\nabla
w_\varepsilon|^2\,dxdy-\frac{t^{2_\alpha^*}}{2_\alpha^*}\int_{\mathbb R^N}
|w_\varepsilon(x,0)|^{2_\alpha^*}\,dx:= \frac {t^2}2 A - \frac{t^{2_\alpha^*}}{2_\alpha^*}B
$$
attains its maximum $\frac{\alpha}{2N}(S_{\alpha}^E)^{\frac
N\alpha}$ at $t_0 = (\frac AB)^{\frac 1{2^*_\alpha-2}}$. Consider the function $f:[0,\infty)\to \mathbb R$, where $f$ is defined by $ f(t)=I(t(\varphi w_\varepsilon))$. Obviously, $f(t)$
attains its maximum at some $t_\varepsilon>0$, and it is standard to see that $t_\varepsilon\to
t_0$ as $\varepsilon \to 0$. In terms of Lemma \ref{lem:3.1.1}, we deduce that
$$
f(t_\varepsilon)\leq \frac {t_0^2}2\int_{\mathbb R_+^{N+1}} |\nabla
w_\varepsilon|^2\,dxdy-\frac{t_0^{2_\alpha^*}}{2_\alpha^*}\int_{\mathbb R^N}
|w_\varepsilon(x,0)|^{2_\alpha^*}\,dx+C\varepsilon^{N-\alpha}+C\varepsilon^N-C\varepsilon^{\frac{(\alpha-N)q}{2}+N-s},
$$
which implies by the assumption $N>\frac{2(\alpha-s)}{q}+\alpha$ that
$$
f(t_\varepsilon)< \frac{\alpha}{2N}(S_{\alpha}^E)^{\frac N\alpha}
$$
for $\varepsilon$ small enough. Since $I$ satisfies the $(PS)_c$
condition for $c\in (0,\frac{\alpha}{2N}(S_{\alpha}^E)^{\frac
N\alpha})$, then it follows from the Mountain Pass theorem that the
functional $I$ possesses at least one nontrivial critical point.

Next, we deal with the case (ii). In this case, it is easy to verify that $I$
possesses the Mountain Pass structure since $0<\mu<\mu_s$. Moreover,
it was proved in \cite{BCP} that $I$ satisfies the $(PS)_c$
condition for $c\in (0,\frac{\alpha}{2N}(S_{\alpha}^E)^{\frac
N\alpha})$. To prove that $I$ possesses a nontrivial critical point,
it is sufficient to show that the Mountain Pass level of $I$ is
below $\frac{\alpha}{2N}(S_{\alpha}^E)^{\frac N\alpha}$. With the
same notations as above, we have
$$
f(t_\varepsilon)\leq \frac {t_0^2}2\int_{\mathbb R_+^{N+1}} |\nabla
w_\varepsilon|^2\,dxdy-\frac{t_0^{2_\alpha^*}}{2_\alpha^*}\int_{\mathbb R^N}
|w_\varepsilon(x,0)|^{2_\alpha^*}\,dx+C\varepsilon^{N-\alpha}+C\varepsilon^N-C\varepsilon^{\alpha-s}.
$$
Since $N>2\alpha-s$, we deduce that
$$
f(t_\varepsilon)<\frac{\alpha}{2N}(S_{\alpha}^E)^{\frac N\alpha}
$$
for $\varepsilon$ small enough. This proves (ii).

\end{proof}

\bigskip
\subsection {Singular term with Hardy-Sobolev critical exponent}
\bigskip
We assume in this subsection that $2\leq r< 2^*_\alpha,\, q = 2^*_\alpha(s)$. We first establish a prior bound for  the ground state solution of
\begin{equation}\label{eq:3.2.1}
(-\Delta)^{\frac \alpha 2} u=\frac{u^{2_\alpha^*(s)-1}}{|x|^s} \quad{\rm in }\quad \mathbb
R^N.
\end{equation}
The extension $v\in \mathcal{D}^{1,2}(y^{1-\alpha},\mathbb{R}^{N+1}_+)$ of the ground state solution $u$ of \eqref{eq:3.2.1} satisfies
\begin{equation}\label{eq:3.2.2}
 \left\{
  \begin{array}{ll}
  \displaystyle
div (y^{1-\alpha} \nabla v)= 0 \quad &{\rm in }\quad \mathbb{R}^{N+1}_+,\\
 \lim_{y\to 0}(y^{1-\alpha}\frac{\partial v}{\partial \nu})= \frac{v^{2_\alpha^*(s)-1}}{|x|^s} \quad &{\rm in }\quad \mathbb{R}^N,    \\
\end{array}
\right.
 \end{equation}
where $\mathcal{D}^{1,2}(y^{1-\alpha},\mathbb{R}^{N+1}_+)$ is defined as the closure of $C^\infty_0(\overline{\mathbb{R}^{N+1}_+})$ under the norm
\[
\|w\|=\bigg( \int_{\mathbb{R}^{N+1}_+}y^{1-\alpha}|\nabla w|^2\,dxdy\bigg)^{\frac12}.
\]
The function $v$ is also related to the minimizer of the variational problem
$$
S_{s,\alpha}^E=\inf_{w\in \mathcal D^{1,2}(y^{1-\alpha}, \mathbb
R_+^{N+1})\setminus
\{0\}}\frac{\int_{\mathbb{R}^{N+1}_+}y^{1-\alpha}|\nabla
w(x,y)|^2\,dxdy}{(\int_{\mathbb{R}^{N}}\frac{|w(x,0)|^{2_\alpha^*(s)}}{|x|^s}\,dx)^{\frac{2}{2_\alpha^*(s)}}}.
$$
To find decaying laws for solutions of \eqref{eq:3.2.2}, we consider the linear problem
\begin{equation}\label{eq:3.2.3}
 \left\{
  \begin{array}{ll}
  \displaystyle
div (y^{1-\alpha} \nabla w)= 0 \quad &{\rm in }\quad \mathbb{R}^{N+1}_+,\\
 \lim_{y\to 0}(y^{1-\alpha}\frac{\partial w}{\partial \nu})= \frac{a(x)}{|x|^s}w \quad &{\rm in }\quad \mathbb{R}^N.    \\
\end{array}
\right.
 \end{equation}

For $D\subset \mathbb{R}^{N+1}_+$, $\partial D$ is the boundary of
$D$ in $\mathbb{R}^{N+1}_+$. Denote $\partial'D = \bar D\cap
\partial \mathbb{R}^{N+1}_+$ and $\partial''D = \partial D\setminus
\partial'D$.
\begin{Lemma}\label{lem:3.2.1}
Suppose $a\in L^{\frac{\alpha - s}{N-s}}(\mathbb{R}^N,\frac 1{|x|^s})$ and $w\in \mathcal{D}^{1,2}(y^{1-\alpha},\mathbb{R}^{N+1}_+), w>0$ is a solution of \eqref{eq:3.2.3}. Then, there exists a positive constant $C = C(N,s,\alpha,\|a\|_{L^{\frac{\alpha - s}{N-s}}(\mathbb{R}^N,\frac 1{|x|^s})})$ such that
\[
\sup_{Q_{\frac 12}}w\leq C \|w\|_{L^2(y^{1-\alpha},Q_R)}
\]
\end{Lemma}

\begin{proof} Denote $Q_R = B_R(y)\times (0,R)$. Suppose $0<r<R\leq 1$. Let $\xi\in C^1_0(Q_1\cup\partial'Q_1)$ with $\xi =1$ in $Q_r$, $\xi =0$ outside $Q_R$ and $|\nabla \xi|\leq \frac 2{R-r}$. Let $k>0$ be any number which is eventually sent to $0$.  Let $q =\frac p{2^*}$, $\bar w = w+k$. Define
\[
 \bar w_m = \left\{
  \begin{array}{ll}
  \displaystyle
\bar w \quad &{\rm if }\quad \bar w<m,\\
 k+m \quad &{\rm if }\quad \bar w\geq m.    \\
\end{array}
\right.
\]
Consider the test function
\[
\eta = \xi^2(\bar w_m^{2q-2}\bar w - k^{2q-1})\in H^1_{0,L}(\mathbb{R}^{N+1}_+).
\]
By \eqref{eq:3.2.3},
\[
\int_{\mathbb{R}^{N+1}_+}y^{1-\alpha}\nabla w\nabla \eta\,dxdy = \int_{\mathbb{R}^{N}}\frac {a(x)}{|x|^s}w\eta\,dx.
\]
We deduce
\[
\begin{split}
&\int_{\mathbb{R}^{N+1}_+}y^{1-\alpha}\nabla w\nabla \eta\,dxdy\\
&= \int_{\mathbb{R}^{N+1}_+}y^{1-\alpha}\xi^2\bar w_m^{2q-2}[(2q-2)|\nabla\bar w_m|^2 + |\nabla w|^2]\,dxdy\\
&+ 2\int_{\mathbb{R}^{N+1}_+}y^{1-\alpha}\xi(\bar w_m^{2q-2}\bar w - k^{2q-1})\nabla \xi\nabla w\,dxdy\\
&\geq \int_{\mathbb{R}^{N+1}_+}y^{1-\alpha}\xi^2\bar w_m^{2q-2}[(2q-2)|\nabla\bar w_m|^2 + |\nabla w|^2]\,dxdy\\
&- \int_{\mathbb{R}^{N+1}_+}y^{1-\alpha}(C_\epsilon|\nabla\xi|^2\bar w_m^{2q-2}w^2 + \epsilon\xi^2\bar w_m^{2q-2}|\nabla w|^2)\,dxdy.\\
\end{split}
\]
On the other hand, by H\"{o}lder's inequality,
\[
\begin{split}
&\int_{\mathbb{R}^{N}}\frac {a(x)}{|x|^s}w\eta\,dx\\
&\leq \Big(\int_{\mathbb{R}^{N}\cap supp \xi}\frac {|a(x)|^{\frac{N-s}{\alpha-s}}}{|x|^s}\,dx\Big)^{\frac{\alpha-s}{N-s}}\Big(\int_{\mathbb{R}^{N}}\frac {|\xi\bar w_m^{q-1}w(x,0)|^{2^*_{s,\alpha}}}{|x|^s}\,dx\Big)^{\frac 2{2^*_{s,\alpha}}}.\\
\end{split}
\]
Hence,
\[
\begin{split}
&\int_{\mathbb{R}^{N+1}_+}y^{1-\alpha}\xi^2\bar w_m^{2q-2}[(2q-2)|\nabla\bar w_m|^2 + |\nabla w|^2]\,dxdy\\
&\leq C_\epsilon\int_{\mathbb{R}^{N+1}_+}y^{1-\alpha}|\nabla\xi|^2\bar w_m^{2q-2}w^2\,dxdy\\
& + \Big(\int_{\mathbb{R}^{N}\cap supp \xi}\frac {|a(x)|^{\frac{N-s}{\alpha-s}}}{|x|^s}\,dx\Big)^{\frac{\alpha-s}{N-s}}\Big(\int_{\mathbb{R}^{N}}\frac {|\xi\bar w_m^{q-1}w(x,0)|^{2^*_{s,\alpha}}}{|x|^s}\,dx\Big)^{\frac 2{2^*_{s,\alpha}}}.\\
\end{split}
\]
Let $U = w_m^{q-1}w$. We have
\[
|\nabla U|^2\leq C(2q-1)[(2q-2)w_m^{2q-2}|\nabla \bar w_m|^2 + w_m^{2q-2}|\nabla \bar w|^2],
\]
and then
\[
\begin{split}
&\int_{\mathbb{R}^{N+1}_+}y^{1-\alpha}\xi^2|\nabla U|^2\,dxdy\\
&\leq C(2q-1)\Big[ \int_{\mathbb{R}^{N+1}_+}y^{1-\alpha}\xi^2|U|^2\,dxdy\\
& + \Big(\int_{\mathbb{R}^{N}\cap supp \xi}\frac {|a(x)|^{\frac{N-s}{\alpha-s}}}{|x|^s}\,dx\Big)^{\frac{\alpha-s}{N-s}}\Big(\int_{\mathbb{R}^{N}}\frac {|\xi U(x,0)|^{2^*_{s,\alpha}}}{|x|^s}\,dx\Big)^{\frac 2{2^*_{s,\alpha}}}\Big].\\
\end{split}
\]
Consequently,
\[
\begin{split}
&\int_{\mathbb{R}^{N+1}_+}y^{1-\alpha}|\nabla(\xi U)|^2\,dxdy\\
&\leq C(2q-1)\Big[ \int_{\mathbb{R}^{N+1}_+}y^{1-\alpha}|\xi U|^2\,dxdy\\
& + \Big(\int_{\mathbb{R}^{N}\cap supp \xi}\frac {|a(x)|^{\frac{N-s}{\alpha-s}}}{|x|^s}\,dx\Big)^{\frac{\alpha-s}{N-s}}\Big(\int_{\mathbb{R}^{N}}\frac {|\xi U(x,0)|^{2^*_{s,\alpha}}}{|x|^s}\,dx\Big)^{\frac 2{2^*_{s,\alpha}}}\Big].\\
\end{split}
\]
By the Sobolev inequality and Hardy-Sobolev inequality,
\[
\begin{split}
&\int_{\mathbb{R}^{N+1}_+}y^{1-\alpha}|\nabla(\xi U)|^2\,dxdy\\
&\geq\frac 12 C\Big(\int_{\mathbb{R}^{N}}\frac {|\xi U(x,0)|^{2^*_{s,\alpha}}}{|x|^s}\,dx\Big)^{\frac 2{2^*_{s,\alpha}}}
 +\frac 12\int_{\mathbb{R}^{N+1}_+}y^{1-\alpha}|\nabla(\xi U)|^2\,dxdy\\
\end{split}
\]
Since $a\in L^{\frac{\alpha - s}{N-s}}(\mathbb{R}^N,\frac 1{|x|^s}\,dx)$, we may choose $R>0$ small so that
\[
\begin{split}
&\int_{\mathbb{R}^{N+1}_+}y^{1-\alpha}|\nabla(\xi U)|^2\,dxdy\leq C(2q-1) \int_{\mathbb{R}^{N+1}_+}y^{1-\alpha}|\xi U|^2\,dxdy\\
\end{split}
\]
Now, the proof can be completed as the proof of Proposition 3.1 in \cite{TX}.

\end{proof}

It was proved in \cite{Yang} that problem \eqref{eq:3.2.2} possesses a
ground state solution $u$, which is radially symmetric and monotonicity decreasing in $|x|$. One may verify that $u_\varepsilon(x)=\varepsilon^{\frac{\alpha-N}2}u(\frac x \varepsilon)$ also solves problem \eqref{eq:3.2.2} for any $\varepsilon>0$. However, it is not known the explicit formula of the ground state solution. Let $w$ be the extension of $u$. To
study critical problem, we need to establish decaying law of $w$ at infinity.
\begin{Lemma}\label{lem:3.2.2}
Suppose $w\in \mathcal{D}^{1,2}(y^{1-\alpha},\mathbb{R}^{N+1}_+)$ is the extension of a
ground state solution $u$ of \eqref{eq:3.2.2}. Then, there exists a
positive constant $C>0$ such that
\[
w(x,y)\leq \frac
C{(1+|X|^2)^{\frac {N-\alpha}2}}
\]
for $X=(x,y)\in  \mathbb{R}^{N+1}_+$.
\end{Lemma}
\bigskip

\begin{proof}
 Consider the Kelvin transformation
\[
\tilde w(X) = |X|^{-N+\alpha}w\bigg(\frac X{|X|^2}\bigg)
\]
of $w$. If $w$ is a
solution of \eqref{eq:3.2.2}, then $\tilde w$ is also a solution of
problem \eqref{eq:3.2.2}. Moreover,  we have
\begin{equation}\label{eq:3.2.5}
\int_{ \mathbb{R}^{N+1}_+}y^{1-\alpha}|\nabla \tilde w|^2\,dxdy\leq
C,\quad \int_{\mathbb{R}^{N}}|\tilde w(x,0)|^{2^*_\alpha}\,dx\leq C.
\end{equation}
In equation \eqref{eq:3.2.2}, we choose $a(x) = \tilde
w^{2^*_\alpha(s)-2}(x,0)$, then $a\in L^{\frac{\alpha -
s}{N-s}}(\mathbb{R}^N,\frac 1{|x|^s})$. By Lemma \ref{lem:3.2.1},
$\tilde w\in L^\infty_{loc}(\mathbb{R}^{N+1}_+)$. It results that
for $X\in Q_{\frac 12}$, we have $|X|^{-N+\alpha}w\bigg(\frac
X{|X|^2}\bigg)\leq C$ for a positive constant $C$. Hence, if $|X|\geq \frac 12$, we have $w(X)\leq C |X|^{\alpha- N}$.
In the same way, we have $w(X)\leq C$ if $|X|\leq \frac 12$. The assertion follows.
\end{proof}
 Now we give a decay estimate for the gradient of the ground state solution $u$.
\begin{Lemma}\label{lem:3.2.3}
Let $u\in \dot{H}^{\frac \alpha 2}(\mathbb{R}^{N} )$ be a positive solution $u$ of \eqref{eq:3.2.1}. Then, there exists a positive constant $C>0$ such that
\begin{equation}\label{eq:3.2.5a1}
u(x)\geq \frac C{(1+|x|^2)^{\frac{N-\alpha}2}}
\end{equation}
if $x\in \mathbb{R}^N$, and
\begin{equation}\label{eq:3.2.5a}
|\nabla u(x)|\leq  C|x|^{-(N+1-\alpha)}
\end{equation}
for $|x|>0$.
\end{Lemma}

\begin{proof} First, we consider the case $|x|\geq 1$. By Lemma \ref{lem:3.2.2} and
Proposition 2.6 in \cite{JLX}, there exists $\beta\in (0,1)$  such
that the extension $w$ of $u$ belongs to $C^{1,\beta}_{loc}(\mathbb{R}^{N+1}_+\setminus Q_1)$. In
particularly, $u\in C^{1,\beta}_{loc}(\mathbb{R}^{N}\setminus
B_1(0))$. Obviously, solutions of \eqref{eq:3.2.1} can be expressed
by Riesz potential as
\begin{equation}\label{eq:3.2.5b}
u(x)=\int_{\mathbb
R^N}\frac{u(y)^{\frac{N+\alpha-2s}{N-\alpha}}}{|y|^s|x-y|^{N-\alpha}}\,dy.
\end{equation}
Therefore, for $|x|\geq 1$,
$$
\frac{\partial u}{\partial x_i}(x)=(\alpha-N) \int_{\mathbb
R^N}\frac{u(y)^{\frac{N+\alpha-2s}{N-\alpha}}}{|y|^s|x-y|^{N-(\alpha-1)}}\frac{x_i-y_i}{|x-y|}\,dy.
$$
We infer from Lemma \ref{lem:3.2.2} that
$$
|\frac{\partial u}{\partial x_i}(x)|\leq C \int_{\mathbb
R^N}\frac{1}{|y|^{N-s+\alpha}|x-y|^{N-(\alpha-1)}}\,dy=C|x|^{(s-1)-N},
$$
where the last equality follows from  page 132 in \cite{LR}. Since
$s<\alpha$, we have $|\frac{\partial u}{\partial x_i}(x)|\leq C
|x|^{(\alpha-1)-N}$, that is, \eqref{eq:3.2.5a} holds in the case $|x|\geq 1$.

Next, we deal with the case $0<|x|\leq 1$, and meanwhile we bound $u(x)$ from below. We note that the Kelvin transform $v(x)= \frac 1{|x|^{N-\alpha}}u(\frac x{|x|^2})$ of $u(x)$ is also a solution of \eqref{eq:3.2.1}, and so  $v$ satisfies \eqref{eq:3.2.5b}. As a consequence of Lemma \ref{lem:3.2.2},
\[
v(x) \leq \frac C{(1+|x|^2)^{\frac{N-\alpha}2}}
\]
for $x\in \mathbb{R}^N$. By Proposition 2.4 in  \cite{JLX}, $v$ is H\"{o}lder continuous. This implies that there exists a positive constant $C>0$ such that $v(x)\geq C$ if $|x|\leq 1$. In other word,
\[
u(x)\geq \frac C{|x|^{N-\alpha}}
\]
if $|x|\geq 1$, \eqref{eq:3.2.5a1} follows. Now, we prove \eqref{eq:3.2.5a} for $0<|x|\leq 1$. We know that
\[
|\nabla v(x)|\leq  C|x|^{-(N+1-\alpha)}
\]
holds for $|x|\geq 1$. We remark that $v$ is radially symmetric. Thus, if $0<r = |x|\leq 1$, we have $\frac 1r\geq 1$
and
\begin{equation}\label{eq:3.2.5c}
v\Big(\frac 1r\Big)\leq C r^{N-\alpha},\quad |v'\Big(\frac 1r\Big)|\leq Cr^{N-\alpha +1}.
\end{equation}
Furthermore, by \eqref{eq:3.2.5c}, for $0<r\leq 1$,
\[
u(r)= \frac 1{r^{N-\alpha}}v\Big(\frac 1{r}\Big)
\]
satisfies
\[
\begin{split}
|u'(r)| &= \Big|(\alpha-N)r^{\alpha-N-1}v\Big(\frac 1r\Big) - r^{\alpha-N-2}v'\Big(\frac 1r\Big)\Big|\\
&\leq Cr^{\alpha-N-1}r^{N-\alpha} + Cr^{\alpha-N-2}r^{N-\alpha +1}\\
&\leq Cr^{-1}.\\
\end{split}
\]
Consequently, if $0<r\leq 1$,
\[
|u'(r)|\leq \frac C{r^{N-\alpha+1}}.
\]
The assertion follows.

\end{proof}

The next lemma concerns some properties of the extension $w$ of
$u$.

\begin{Lemma}\label{lem:3.2.4}
Let $w$ be the extension of $u$.

(i) If $0<\alpha<2$, then $|\nabla w(x,y)|\leq \frac C yw(x,y)\leq
\frac C{y|(x,y)|^{N-\alpha}}$.

(ii) If $1\leq \alpha <2$, then $|\nabla w(x,y)|\leq \frac
C{(|x|^2+y^2)^{\frac{N-(\alpha-1)}{2}}}$.

(iii) Let $w_\varepsilon$ be the $\alpha$-extension of
$u_\varepsilon$, then $w_\varepsilon=\varepsilon^{\frac
{\alpha-N}2}w(\frac x\varepsilon,\frac y\varepsilon)$.
\end{Lemma}

\begin{proof}
We denote $P_y^\alpha(x)=\frac
{y^\alpha}{(|x|^2+y^2)^{\frac{N+\alpha}2}}$ the Possion kernel of
$(-\Delta)^{\frac \alpha 2}$, then
$$
w(x,y)=\int_{\mathbb R^N} \frac {y^\alpha}{(|x-z|^2+y^2)^{\frac
{N+\alpha}2}}u(z)\,dz.
$$
Direct calculation shows that for $i = 1,\cdots, N$,
\begin{eqnarray*}
\bigg|\frac{\partial w(x,y)}{\partial x_i}\bigg|&\leq &(N+\alpha)\int_{\mathbb R^N}
\frac{y^\alpha|x-z|}{(|x-z|^2+y^2)^{\frac {N+\alpha}2+1}}u(z)\,dz\\
&\leq & \frac{N+\alpha}{2y}\int_{\mathbb R^N} \frac
{y^\alpha}{(|x-z|^2+y^2)^{\frac {N+\alpha}2}}u(z)\,dz\\&=&\frac C y
w(x,y)
\end{eqnarray*}
Similarly, we have
\begin{eqnarray*}
\bigg|\frac{\partial w(x,y)}{\partial y}\bigg|&=&\bigg|\int_{\mathbb R^N} \frac{\alpha
y^{\alpha-1}}{(|x-z|^2+y^2)^{\frac
{N+\alpha}2}}u(z)\,dz\\
&-&(N+\alpha)\int_{\mathbb R^N}
\frac{y^{\alpha+1}}{(|x-z|^2+y^2)^{\frac {N+\alpha}2+1}}u(z)\,dz\bigg|\\
&=& \bigg|\int_{\mathbb R^N} \frac
{y^{\alpha-1}[\alpha|x-z|^2-Ny^2]}{(|x-z|^2+y^2)^{\frac
{N+\alpha}2+1}}u(z)\,dz\bigg|\\ &\leq & C\int_{\mathbb R^N} \frac
{y^{\alpha-1}}{(|x-z|^2+y^2)^{\frac {N+\alpha}2}}u(z)\,dz\\
&=&\frac Cy w(x,y).
\end{eqnarray*}
(i) follows by Lemma \ref{lem:3.2.2}.

For (ii), if we denote $P(x)=\frac 1{[1+|x|^2]^{\frac{N+\alpha}2}}$,
then $P_y^\alpha(x)=\frac 1{y^N}P(\frac xy)$. So we have
$$
w(x,y)=\int_{\mathbb R^N} \frac 1{y^N}P\bigg(\frac {x-z}y\bigg)u(z)\,dz.
$$
Hence, we have
\begin{eqnarray*}
\bigg|\frac{\partial w(x,y)}{\partial y}\bigg|&=&\bigg|\frac{\partial }{\partial y} \int_{\mathbb R^N} \frac
1{y^N}P(\frac {x-z}y)u(z)\,dz\bigg|\\&=&\bigg|\frac{\partial}{\partial y} \int_{\mathbb R^N}
P(\tilde z)u(x-y\tilde z)\,d \tilde z\bigg|\\&=& \bigg|\int_{\mathbb
R^N}P(\tilde z)\langle \nabla u(x-y\tilde z),\tilde z
\rangle\,d\tilde z\bigg|\\&=&\bigg|\int_{\mathbb R^N} P(\frac{x-z}y)\frac
1{y^N}\langle \frac{x-z}y,\nabla u(z)\rangle\,dz \bigg|\\
\end{eqnarray*}
By Lemma \ref{lem:3.2.3},
\begin{eqnarray*}
\bigg|\frac{\partial w(x,y)}{\partial y}\bigg|&\leq &|\int_{\mathbb R^N} \frac 1{y^N}P(\frac
{x-z}y)\frac{|x-z|}{y}\frac 1{|z|^{N-\alpha+1}}\,dz|\\&\leq&
|\int_{\mathbb R^N}
\frac{y^{\alpha}}{y^{N+\alpha}(1+\frac{|x-z|^2}{y^2})^{\frac
{N+\alpha}2}} \frac{|x-z|}y \frac
1{|z|^{N-\alpha+1}}\,dz|\\&=&|\int_{\mathbb R^N}
\frac{y^{\alpha-1}}{(y^2+|x-z|^2)^{\frac{N+\alpha-1}2}}\frac
1{|z|^{N-\alpha+1}}\,dz|\\&=& \frac
C{(|x|^2+y^2)^{\frac{N-(\alpha-1)}2}},
\end{eqnarray*}
where the last inequality holds since $\alpha\geq 1$ and the
$(\alpha-1)$-extension of $\frac 1{|x|^{N-\alpha+1}}$ is $\frac
1{(|x|^2+y^2)^{\frac{N-(\alpha-1)}2}}$. Similarly, we have
\begin{eqnarray*}
\bigg|\frac{\partial w(x,y)}{\partial x_i}\bigg|&\leq &|\int_{\mathbb R^N}
P(\frac{x-z}y)\frac 1{y^N}| \nabla u(z)|\,dz |\\&\leq
&|\int_{\mathbb R^N} \frac 1{y^N}P(\frac {x-z}y)\frac
1{|z|^{N-\alpha+1}}\,dz|
\\&=&|\int_{\mathbb R^N} P_y^\alpha(x-z)\frac
1{|z|^{N-\alpha+1}}\,dz|\\&\leq& \int_{\mathbb
R^N}\frac{y^{\alpha-1}}{[|x-z|^2+y^2]^{\frac{N+(\alpha-1)}2}}\frac{1}{|z|^{N-(\alpha-1)}}\,dz
\\&=& \frac
C{(|x|^2+y^2)^{\frac{N-(\alpha-1)}2}}.
\end{eqnarray*}
This proves (ii).

(iii) follows from the scaling invariance of $u_\varepsilon$ and $P$.

\end{proof}

Let $\varphi(x,y)$ be defined as in \eqref{eq:3.1.1a}, and let $u>0$ be a ground state solution of \eqref{eq:3.2.1},
which is radially symmetric. Denote by $w$ the extension of $u$. Set $u_\varepsilon(x) = \varepsilon^{\frac{\alpha - N}2}u(\frac x\varepsilon)$. Then, the extension $w_\varepsilon(x)$ of $u_\varepsilon(x)$ is given by $w_\varepsilon(x,y) = \varepsilon^{\frac{\alpha - N}2}w(\frac x\varepsilon, \frac y\varepsilon)$.

\begin{Lemma}\label{lem:3.2.5}
We have the following estimates

(i) $\|\varphi w_\varepsilon\|_{H_{0,L}^1(\mathcal{C}_\Omega)}^2\leq \|w_\varepsilon\|_{H_{0,L}^1(\mathcal{C}_\Omega)}^2+O(\varepsilon^{N-\alpha})$.

(ii) If $N> 2\alpha$, then $\|\varphi u_\varepsilon\|_{L^2(\Omega)}^2\geq
C\varepsilon^\alpha$, and if $N= 2\alpha$, $\|\varphi
u_\varepsilon\|_{L^2(\Omega)}^2\|\geq C\varepsilon^\alpha |\ln\varepsilon|$;

(iii) $\int_{\Omega} \frac{|\varphi
u_\varepsilon|^{2_\alpha^*(s)}}{|x|^s}\,dx=\int_{\mathbb R^N}
\frac{| u_\varepsilon|^{2_\alpha^*(s)}}{|x|^s}\,dx+O(\varepsilon^{N-s})$.

(iv) $\int_{\Omega} \frac{|\varphi u_\varepsilon|^{q}}{|x|^s}\,dx\geq
C\varepsilon^{\frac{\alpha-N}{2}q+N-s}$ for $q<2_\alpha^*(s)$.

(v) $\int_{\Omega} |\varphi u_\varepsilon|^r\geq
C\varepsilon^{\frac{(\alpha-N)r}{2}+N}$ for
$2<r<\frac{2N}{N-\alpha}$.

\end{Lemma}

\begin{proof}
We prove (i) first. Note that
\begin{eqnarray*}
\|\varphi w_\varepsilon\|_{H_{0,L}^1(\mathcal{C}_\Omega)}^2&=&\int_{\mathcal{C}_\Omega}
y^{1-\alpha}(\varphi^2|\nabla w_\varepsilon|^2+|w_\varepsilon|^2|\nabla
\varphi|^2+2w_\varepsilon\varphi\langle \nabla w_\varepsilon, \nabla
\varphi\rangle)\,dxdy\\&\leq& \|w_\varepsilon\|_{H_{0,L}^1(\mathcal{C}_\Omega)}^2+\int_{\mathcal{C}_\Omega}
y^{1-\alpha}|w_\varepsilon|^2|\nabla \varphi|^2\,dxdy+2\int_{\mathcal{C}_\Omega}
y^{1-\alpha}w_\varepsilon\varphi |\nabla w_\varepsilon||\nabla
\varphi|\,dxdy.
\end{eqnarray*}
By Lemma \ref{lem:3.2.2},
\begin{eqnarray*}
&&\int_{\mathcal{C}_\Omega} y^{1-\alpha}|w_\varepsilon|^2|\nabla \varphi|^2\,dxdy\\&\leq &C\int_{\{\frac
r2\leq \rho_{xy}\leq \rho\}}y^{1-\alpha}w_\varepsilon^2\\&\leq &C
\varepsilon^{N-\alpha}\int_{\{\frac \rho2\leq \rho_{xy}\leq
\rho\}}y^{1-\alpha}\rho_{xy}^{2(\alpha-N)}\,dxdy\\&=& O(\varepsilon^{N-\alpha}).
\end{eqnarray*}
Similarly,
\begin{eqnarray*}
&&\int_{\mathcal{C}_\Omega} y^{1-\alpha}w_\varepsilon\varphi |\nabla w_\varepsilon||\nabla \varphi|\,dxdy\\
&\leq & C\int_{\{\frac \rho2\leq \rho_{xy}\leq
\rho\}} y^{1-\alpha}w_\varepsilon|\nabla w_\varepsilon|\,dxdy\\
&\leq&\varepsilon^{\alpha-N-1}\int_{\{\frac \rho2\leq \rho_{xy}\leq
\rho\}}y^{1-\alpha}|w(\frac x\varepsilon,\frac y\varepsilon)||\nabla w(\frac x\varepsilon,\frac y\varepsilon)|\,dxdy\\
&=& \varepsilon\int_{\{\frac \rho{2\varepsilon}\leq \rho_{xy}\leq\frac
\rho\varepsilon\}}y^{1-\alpha}w(x,y)|\nabla w(x,y)|\,dxdy.
\end{eqnarray*}
For $(x,y)\in \{(x,y):\frac \rho{2\varepsilon}\leq \rho_{xy}\leq\frac
\rho\varepsilon\}$, we have
$$
w(x,y)\leq \frac C{(\frac \rho\varepsilon)^{N-\alpha}} =C\varepsilon^{N-\alpha}.
$$
By (i) of Lemma \ref{lem:3.2.4}, if $0<\alpha<1$,
$$
\int_{\mathcal{C}_\Omega} y^{1-\alpha}w_\varepsilon\varphi |\nabla w_\varepsilon||\nabla \varphi|\,dxdy\leq
\varepsilon^{2(N-\alpha)+1}\int_{\{\frac \rho{2\varepsilon}\leq \rho_{xy}\leq\frac
\rho\varepsilon\}}
y^{-\alpha}\,dxdy=O(\varepsilon^{N-\alpha}).
$$
If $1\leq \alpha<2$, we deduce from (ii) of Lemma \ref{lem:3.2.4} that
$$
\int_{\mathcal{C}_\Omega} y^{1-\alpha}w_\varepsilon\varphi |\nabla w_\varepsilon||\nabla \varphi|\,dxdy\leq
\varepsilon^{2(N-\alpha+1)}\int_{\{\frac \rho{2\varepsilon}\leq \rho_{xy}\leq\frac
\rho\varepsilon\}}
y^{1-\alpha}\,dxdy=O(\varepsilon^{N-\alpha}).
$$
The assertion (i) follows.

Next, we show (ii). By Lemma \ref{lem:3.2.3},
\begin{eqnarray*}
\int_{\Omega}|\varphi u_\varepsilon|^2\,dx&\geq & \int_{\{|x|<\frac
\rho2\}}u_\varepsilon^2\,dx\\&=& \int_{\{|x|\leq
\varepsilon\}}u_\varepsilon^2\,dx+\int_{\{\varepsilon\leq |x|\leq
\frac \rho2\}}u_\varepsilon^2\,dx\\&= &\int_{\{|x|\leq
\varepsilon\}}\varepsilon^{\alpha-N}u\Big(\frac
x\varepsilon\Big)^2\,dx+\int_{\{\varepsilon\leq |x|\leq \frac
\rho2\}}\varepsilon^{\alpha-N}u\Big(\frac
x\varepsilon\Big)^2\,dx\\&=&
C\varepsilon^\alpha+\varepsilon^\alpha\int_{\{1<|x|<\frac
\rho{2\varepsilon}\}}u^2(x)\,dx\\&\geq &
C\varepsilon^\alpha+C\varepsilon^\alpha\int_1^{\frac
\rho{2\varepsilon}}\frac 1{t^{N-2\alpha+1}}\,dt.
\end{eqnarray*}
If $N> 2\alpha$, we have
$$
\int_1^{\frac r{2\varepsilon}}\frac
1{t^{N-2\alpha+1}}\,dt=C\varepsilon^{N-2\alpha}-C',
$$
which implies
$$
\int_{\Omega}|\varphi u_\varepsilon|^2\,dx \geq
C\varepsilon^{\alpha}+C\varepsilon^{\alpha}[\varepsilon^{N-2\alpha}-1]=O(\varepsilon^\alpha).
$$
If $N=2\alpha$, the fact
$$
\int_1^{\frac r{2\varepsilon}}\frac
1{t^{N-2\alpha+1}}\,dt=|\ln\varepsilon|+\ln \frac r2
$$
yields
$$
\int_{\Omega}|\varphi u_\varepsilon|^2\,dx \geq
C\varepsilon^{\alpha}+C\varepsilon^{\alpha}[\ln \frac r2+|\ln
\varepsilon|]=O(\varepsilon^\alpha|\ln \varepsilon|).
$$
This proves (ii).

Now, for (iii), since $s<\alpha$, we have
\begin{eqnarray*}
&&\Big|\int_{\Omega} \frac{| u_\varepsilon|^{\frac{2(N-s)}{N-\alpha}}}{|x|^s}\,dx-\int_{\mathbb{R}_+^N}
\frac{|\varphi u_\varepsilon|^{\frac{2(N-s)}{N-\alpha}}}{|x|^s}\,dx\Big|\\
&\leq&\int_{\{|x|\geq \frac \rho2\}}\frac{|
u_\varepsilon|^{\frac{2(N-s)}{N-\alpha}}}{|x|^s}\,dx\\
&=&\int_{\{|x|>\frac \rho2\}}\frac{\varepsilon^{-(N-s)}|u(\frac x\varepsilon)|^{\frac{2(N-s)}{N-\alpha}}}{|x|^s}\,dx\\
&=& \int_{\{|x|\geq \frac \rho{2\varepsilon}\}}\frac{|u(x)|^{\frac{2(N-s)}{N-\alpha}}}{|x|^s}\,dx\\
&\leq& C\int_{\{|x|\geq \frac \rho{2\varepsilon}\}}\frac
1{|x|^{2N-s}}\,dx\\&=&C\varepsilon^{N-s}.
\end{eqnarray*}
Similarly, (iv) follows by
\begin{eqnarray*}
&&\int_{\mathbb R^N} \frac{|\varphi u_\varepsilon|^{q}}{|x|^s}\,dx
\\&\geq& \int_{\{|x|<\frac
\rho2\}}\frac{\varepsilon^{\frac{(\alpha-N)q}{2}}|u(\frac
x\varepsilon)|^q}{|x|^s}\,dx\\&=&\varepsilon^{\frac{(\alpha-N)q}{2}+N}\int_{\{|x|<\frac
\rho{2\varepsilon}\}}\frac{|u(x)|^q}{|x|^s}\,dx\\
&\geq&C\varepsilon^{\frac{\alpha-N}{2}q+N-s}.
\end{eqnarray*}

Finally, a direct calculation shows that
\begin{eqnarray*}
\int_{\Omega} |\varphi u_\varepsilon|^r\,dx &\geq&\int_{|x|<\frac \rho
2}|u_\varepsilon|^r\\&=&\int_{\{|x|<\frac \rho
2\}}\varepsilon^{\frac{(\alpha-N)r}{2}}|u\Big(\frac
x\varepsilon\Big)|^r\,dx\\&=&\int_{\{|x|<\frac
\rho{2\varepsilon}\}}\varepsilon^{\frac{(\alpha-N)r}{2}+N}|u(x)|^r\,dx\\&\geq
& C\varepsilon^{\frac{(\alpha-N)r}{2}+N},
\end{eqnarray*}
which implies (v). The proof is complete.
\end{proof}

\begin{proof}[{\bf Proof of $(i)$ and $(ii)$ of  Theorem \ref{thm:1.3}:}] \ \ Without loss of generality, we
may assume that $\mu=1$. Let
\[
c_\infty = \inf\{I_\infty(w): w\in \mathcal{D}^{1,2}(y^{1-\alpha},\mathbb{R}^{N+1}_+)\setminus\{0\}, \langle I_{\infty}'(w), w\rangle = 0\},
\]
where
\[
I_\infty(w) = \frac 12\int_{\mathbb{R}^{N+1}_+}y^{1-\alpha}|\nabla w(x,y)|^2\,dxdy-\frac
1{2_\alpha^*(s)}\int_{\mathbb{R}^{N}}\frac{|w(x,0)|^{2_\alpha^*(s)}}{|x|^s}\,dx.
\]
Since $S_{s,\alpha}^E$ is achieved by a radially symmetric function
$u$, $c_\infty$ is achieved by the extension $w$ of $u$, and $w$ is
the ground state solution of \eqref{eq:3.2.2}.  Hence,
$$
c_\infty=\frac{\alpha-s}{2(N-s)}(S_{s,\alpha}^E)^{\frac{N-s}{\alpha-s}}.
$$
We claim that
$I$ satisfies $(PS)_c$ condition for $c\in (0,c_\infty)$, that is, for $c\in (0,c_\infty)$ and any sequence $\{w_n\}\subset H_{0,L}^1(\mathcal{C}_\Omega)$ such that
\begin{equation}\label{eq:3.2.6}
I(w_n)\to c,\quad I'(w_n)\to 0,
\end{equation}
$\{w_n\}$ possesses a convergent subsequence. We first show that $\{w_n\}$ is uniformly bounded
 in $ H_{0,L}^1(\mathcal{C}_\Omega)$. It is easily done for $r = 2$ since $\lambda<\lambda_1$.
  For the case $2<r<2_\alpha^*=\frac{2N}{N-\alpha}$, by \eqref{eq:3.2.6}, we deduce that if $r<2_\alpha^*(s)$,
$$
(\frac 12-\frac 1r)\int_{\mathcal{C}_\Omega}y^{1-\alpha}|\nabla
w(x,y)|^2\,dxdy+(\frac 1r-\frac
1{2_\alpha^*})\int_{\Omega}\frac{|w(x,0)|^{2_\alpha^*(s)}}{|x|^s}\,dx\leq
c+o(1)\|w_n\|_{H_{0,L}^1(\mathcal{C}_\Omega)},
$$
and if
$r\geq 2_\alpha^*(s)$, we have
$$
(\frac 12-\frac 1{2_\alpha^*})\int_{\mathcal{C}_\Omega}y^{1-\alpha}|\nabla
w(x,y)|^2\,dxdy+(\frac 1{2_\alpha^*}-\frac
1r)\int_{\Omega}|w(x,0)|^{r}\,dx\leq c+o(1)\|w_n\|_{H_{0,L}^1(\mathcal{C}_\Omega)}.
$$
Hence, in both cases, $\{\|w_n\|_{H_{0,L}^1(\mathcal{C}_\Omega)}\}$ is bounded. We may show that $\{w_n\}$ is tight as \cite{BCP}, and we may assume that, up to a subsequence,
\[
w_n\rightharpoonup w\quad {\rm in}\quad
{H_{0,L}^1(\mathcal{C}_\Omega)}\,\ {\rm and}\,\ L^q(\Omega, \frac
1{|x|^s}), \quad w_n(x,0)\to w(x,0)\quad{\rm in}\quad L^r(\Omega).
\]
Now we show  that $w_n\to w$ strongly in
$H_{0,L}^1(\mathcal{C}_\Omega)$. Suppose on the contrary,
$z_n=w_n-w\not \to 0$ in $H_{0,L}^1(\mathcal{C}_\Omega)$. It follows
from the Brezis-Lieb Lemma \cite{BL} that
$$
\frac 12\int_{\mathcal{C}_\Omega}y^{1-\alpha}|\nabla z_n(x,y)|^2\,dxdy-\frac
1{2_\alpha^*(s)}\int_{\Omega}\frac{|z_n(x,0)|^{2_\alpha^*(s)}}{|x|^s}\,dx=
c+o(1)-I(w)
$$
and
$$
\int_{\mathcal{C}_\Omega}y^{1-\alpha}|\nabla z_n(x,y)|^2\,dxdy-
\int_\Omega\frac{|z_n(x,0)|^{2_\alpha^*(s)}}{|x|^s}\,dx= o(1).
$$
By the definition of $S_{s,\alpha}^E$,
\[
\begin{split}
&\int_{\mathcal{C}_\Omega}y^{1-\alpha}|\nabla z_n(x,y)|^2\,dxdy\\
&\geq
S_{s,\alpha}^E \big(
\int_{\Omega}\frac{|z_n(x,0)|^{2_\alpha^*(s)}}{|x|^s}\,dx \big
)^{\frac
2{2_\alpha^*}}\\&=S_{s,\alpha}^E(\int_{\mathcal{C}_\Omega}y^{1-\alpha}|\nabla
z_n(x,y)|^2\,dxdy+o(1))^{\frac 2{2_\alpha^*(s)}},
\end{split}
\]
which implies by the fact $z_n\not \to 0$ that
$$
\int_{\mathcal{C}_\Omega}y^{1-\alpha}|\nabla z_n(x,y)|^2\,dxdy\geq
(S_{s,\alpha}^E)^{\frac{N-s}{\alpha-s}}+o(1).
$$
We then conclude that
$$
c+o(1)-I(w)=(\frac
12-\frac{1}{2_\alpha^*(s)})\int_{\mathcal{C}_\Omega}y^{1-\alpha}|\nabla
z_n(x,y)|^2\,dxdy+o(1)\geq
\frac{\alpha-s}{2(N-s)}(S_{s,\alpha}^E)^{\frac{N-s}{\alpha-s}}+o(1),
$$
this contradicts to that $c\in
(0,\frac{\alpha-s}{2(N-s)}(S_{s,\alpha}^E)^{\frac{N-s}{\alpha-s}})$
since $I(w)\geq 0$. Hence, we have $w_n\to w$, i.e., the $(PS)_c$
condition holds.

Next, it is easy to see that the functional $I$ has the Mountain
Pass structure. In order to find critical points of $I$, it is
sufficient to show that the Mountain Pass level of $I$ is below
$c_\infty$. For this purpose, we define
$$
f_\infty(t)=
\frac{t^2}2\|w_\varepsilon\|_{\mathcal{D}^{1,2}(y^{1-\alpha},\mathbb{R}^{N+1}_+)}^2-\frac{t^{2_\alpha^*(s)}}{2_\alpha^*(s)}\int_{\mathbb
R^N} \frac {|u_\varepsilon|^{2_\alpha^*(s)}}{|x|^s}\,dx.
$$
We may show that the maximum of $f_\infty$ is achieved at $\bar t>0$ and
$$
f_\infty(\bar t)=c_\infty.
$$
Moreover, $\bar t>0$ can be worked out explicitly, see arguments in
the proof of Theorem \ref{thm:1.4}. Let $f(t)=I(t(\varphi
w_\varepsilon))$. The maximum of the function $f(t)$ is also
achieved at some $t_\varepsilon>0$, and $t_\varepsilon\to \bar t$ as
$\varepsilon\to 0$. If $N\geq 2\alpha$, we infer from Lemma
\ref{lem:3.2.5} that
$$
f(t_\varepsilon)\leq \frac
{t^2_\varepsilon}2(\|w_\varepsilon\|_{H_{0,L}^1(\mathcal{C}_\Omega)}^2+O(\varepsilon^{N-\alpha}))-C\frac
\lambda r t_\varepsilon^r
\varepsilon^{\frac{(\alpha-N)r}{2}+N}-\frac{t_\varepsilon^{2_\alpha^*(s)}}{2_\alpha^*(s)}(\int_{\Omega}
\frac
{(u_\varepsilon)^{2_\alpha^*(s)}}{|x|^s}+O(\varepsilon^{N-s}))\,dx.
$$
So we have
$$
f(t_\varepsilon)\leq \frac{\bar t^2}2\|w_\varepsilon\|_{H_{0,L}^1(\mathcal{C}_\Omega)}^2-\frac{\bar
t^{2_\alpha^*(s)}}{2_\alpha^*(s)}\int_{\mathbb R^N} \frac
{(u_\varepsilon)^{2_\alpha^*(s)}}{|x|^s}\,dx+C\varepsilon^{N-\alpha}-C\varepsilon^{\frac{(\alpha-N)r}{2}+N}-C\varepsilon^{N-s}.
$$
Since $N>\alpha+\frac {2\alpha}r$, we can choose $\varepsilon>0$ small enough
such that
$$
f(t_\varepsilon)<c_\infty.
$$
This proves (i) of Theorem \ref{thm:1.3}

Now, we prove (ii) of Theorem \ref{thm:1.3}. If $N>2\alpha$, we have as before that
$$
f(t_\varepsilon)\leq \frac{\bar t^2}2\|w_\varepsilon\|^2-\frac{\bar
t^{2_\alpha^*(s)}}{2_\alpha^*(s)}\int_{\mathbb R^N} \frac
{(u_\varepsilon)^{2_\alpha^*(s)}}{|x|^s}\,dx+C\varepsilon^{N-\alpha}-C\varepsilon^{\alpha}-C\varepsilon^{N-s},
$$
which yields
$$
f(t_\varepsilon)<c_\infty
$$
for $\varepsilon$ small enough.

If $N=2\alpha$, similarly we have
$$
f(t_\varepsilon)\leq \frac{\bar t^2}2\|w_\varepsilon\|^2-\frac{\bar
t^{2_\alpha^*(s)}}{2_\alpha^*(s)}\int_{\mathbb R^N} \frac
{(u_\varepsilon)^{2_\alpha^*(s)}}{|x|^s}\,dx+C\varepsilon^{N-\alpha}-C\varepsilon^{\alpha}|\ln
\varepsilon|-C\varepsilon^{N-s},
$$
implying
$$
f(t_\varepsilon)<\infty
$$
for $\varepsilon>0$ small enough. The proof is complete.

\end{proof}

\vspace{2mm} \noindent{\bf Acknowledgment} J. Yang  was supported by
NNSF of China, No:11271170, 11371254;  and GAN PO 555 program of Jiangxi.


\begin{thebibliography}{00}
\frenchspacing

\bibitem{BCP}B.Barrios, E.Colorado, A.De Pablo and  U.S\'{a}nchez, On some critical problems for the fractional Laplacian oprator, J. Diff. Equa., 252(2012) 6133--6162.

\bibitem{BCPa}C.Br\"{a}ndle, E.Colorado and A.De Pablo , A concave-convex elliptic problem involving the  fractional Laplacian, Proc. Roy. Soc. Edinburgh Sect. A, in press.

\bibitem{BL} H. Br\'{e}zis and E. Lieb, A relation between pointwise convergence of functions and convergence of functionals, {\it Proc. Amer. Math, Soc.,} 88 (1983), 486-490.

\bibitem{BN} H. Br\'{e}zis and L. Nirenberg, Positive solutions of nonlinear elliptic equations
involving critical Sobolev exponents, {\it Comm. Pure Appl. Math.,} 36 (1983), 437-477.

\bibitem{CaT} X. Cabr\'{e} and J. Tan, Positive solutions of nonlinear problems involving the square
root of the Laplacian, {\it Adv. in Math.},  224(2010) 2052--2093.

\bibitem{CDDS} A. Capella, J. D\'{a}vila, L. Dupaigne and Y. Sire, Regularity of radial extremal solutions for some non local semilinear equations,
{\it Comm. in Part. Diff. Equa.}, 36(2011) 1353--1384.

\bibitem{CS} L. Caffarelli and L. Silvestre, An extention problem related to the fractional Laplacian, {\it Comm. in Part. Diff. Equa.}, 32(2007) 1245--1260.

\bibitem{CHP}   Daomin Cao, Xiaoming He and Shuangjic Peng,  Positive solutions for some singular critical growth nonlinear elliptic equations. {\it Nonlinear Anal.} 60 (2005), no. 3, 589¨C609.

\bibitem{CLO}  W. Chen, and C. Li and B. Ou, Classification of solutions for an integral equation, {\it Comm. Pure Appl. Math.,} 59 (2006), 330-343.

\bibitem {CT} A. Cotsiolis and N.K. Travoularis,   Best constants for Sobolev inequalities for higher order fractional derivatives, {\it J. Math. Anal. Appl.,} 295 (2004), 225-236.

\bibitem{FW} M. M. Fall and T. Weth, Nonexistence results for a class of fractional elliptic boundary value problems,
 {\it J. Funct. Anal.}, 263 (2012) 2205--2227.

\bibitem{FS} R. L. Frank and R. Seiringer, Non-linear ground state representations and sharp Hardy inequalities,   {\it Jour. Funct. Anal.}, 255(2008), 3407-3430.

\bibitem{GY} N. Ghoussoub and C.Yuan, Multiple solutions for quasi-linear PDEs
involving the critical Sobolev and Hardy exponents, {\it Trans.
Amer. Math. Soc.}, 352 (2000) 5703--5743.

\bibitem{JLX} T. Jin, Y. Li and J. Xiong, On a fractional Nirenberg problem, part I: blow up
analysis and compactness of solutions, arXiv:1111.1332v1 [math.AP] 5 Nov 2011.

\bibitem{KP} Dongsheng Kang and Shuangjie Peng,  Existence of solutions for elliptic problems with critical Sobolev-Hardy exponents. {\it Israel J. Math.} 143 (2004), 281--297.

\bibitem{KP1} Dongsheng Kang  and Shuangjie Peng, Singular elliptic problems in R N   with critical Sobolev-Hardy exponents. Nonlinear Anal. 68 (2008), no. 5, 1332¨C1345.

\bibitem{LP}   Gongbao Li and Shuangjie Peng,  Remarks on elliptic problems involving the Caffarelli-Kohn-Nirenberg inequalities. {\it Proc. Amer. Math. Soc.} 136 (2008), no. 4, 1221¨C1228.

\bibitem {LW} Lin, Chang-Shou and H. Wadade, Minimizing problems for the Hardy-Sobolev type inequality with the singularity on the boundary, {\it Tohoku Math. J.}, 64 (2012), 79--103.

\bibitem{LR} E. H. Lieb and M. Loss, Analysis, {\it Graduate Studies in Mathematicas 14,} AMS, 2001.

\bibitem{L} P. Lions, The concentration compactness principle in the calculus of variations. The limit
case (Part 1 and Part 2), {\it Rev. Mat. Iberoamericana}, 1 (1985) 145-201, 45-121.

\bibitem{LZ}
G. Lu and J.Zhu, Symmetry and regularity of extremals of an integral
equation related to the Hardy-Sobolev inequality, {\it Calc. Var.
Partial Differential Equations}, 42 (2011), 563--577.

\bibitem{T}J. Tan, The Brezis Nirenberg type problem involving the square root of the
Laplacian, {\it Calc.  Var. Partial Differential Equations}, 42(2011) 21--41.

\bibitem{NPV} E. Di Nezza, G. Palatucci and E. Valdinoci, Hitchhiker's guide to the fractional Sobolev spaces, {\it Bull. Sci. Math.}, 229(2012), 521-573.

\bibitem{Ra}
P.H.Rabinowitz, \emph{Minimax Methods in Critical Point Theory with
Applications to Differential Equations}, CBMS Regional Conference,
65, AMS, Providence, R.I., 1986.

\bibitem{TX}
J. Tan and J. Xiong, A Harnack inequality for fractional Laplace
equations with lower order terms, {\it Discrete Contin. Dyn. Syst.},
31 (2011) 975--983.

\bibitem{Y}  D. Yafaev, Sharp constants in the Hardy-Rellich inequalities, {\it J. Funct. Anal,} 168 (1999), 121-144.

\bibitem{Yang}
J. Yang, Fractional Sobolev-Hardy inequality in $\mathbb R^N$, {\it
Nonlinear Analysis: Theory, Methods $\&$ Applications}, in press.




\end{thebibliography}
\end{document}